\documentclass{amsart}

\usepackage[utf8]{inputenc}
\usepackage{scrtime}

\usepackage{amsmath,amssymb,amscd,latexsym,color}
\usepackage{verbatim}
\usepackage{graphicx}
\usepackage{mathtools}
\mathtoolsset{showonlyrefs=true} 

\usepackage[english]{babel}

\newcommand{\regine}[1]{\marginpar{{\footnotesize R: #1}}}

\newcommand{\N}{\mathbb N}

\newcommand{\Z}{\mathbb{Z}}

\newcommand{\Zd}{\mathbb{Z}^d}

\newcommand{\Q}{\mathbb{Q}}

\newcommand{\R}{\mathbb{R}}

\renewcommand{\P}{\mathbb{P}}
\newcommand{\E}{\mathbb{E}}

\renewcommand{\epsilon}{\varepsilon}
\renewcommand{\phi}{\varphi}
\renewcommand{\limsup}{\overline{\lim}}

\newcommand{\miniop}[3]{%
\renewcommand{\arraystretch}{0.6}
\begin{array}{c}
{\scriptstyle #1}\\
#2\\
{\scriptstyle #3}
\end{array}
\renewcommand{\arraystretch}{1}}

\newcommand{\1}{1\hspace{-1.3mm}1}

\begin{document}

{
\newtheorem{theorem}{Theorem}[section]
\newtheorem{conjecture}[theorem]{Conjecture}

}
\newtheorem{conjecture-dsk}[theorem]{Conjecture}

\newtheorem{lemme}[theorem]{Lemma}
\newtheorem{defi}[theorem]{Definition}
\newtheorem{coro}[theorem]{Corollary}
\newtheorem{rema}[theorem]{Remark}
\newtheorem{propo}[theorem]{Proposition}
\newtheorem{hyp}{Assumptions}

\newcommand{\T}[2]{{#1}.{#2}} 

\title[Critical branching random walk in random environment]{The critical branching random walk in a random environment dies out}

{
\author{Olivier Garet}
\address{Institut \'Elie Cartan Nancy (math{\'e}matiques)\\
Universit{\'e} de Lorraine\\
Campus Scientifique, BP 239 \\
54506 Vandoeuvre-l{\`e}s-Nancy  Cedex France\\}
\email{Olivier.Garet@univ-lorraine.fr}
\author{R{\'e}gine Marchand}
\email{Regine.Marchand@univ-lorraine.fr}

}

\def\motsclefs{branching random walk, random environment, survival, critical behavior, renormalization,  block construction.}


\begin{abstract}
We study the possibility  for  branching random walks in random environment (BRWRE) to survive. 
The particles perform simple symmetric random 
walks on the $d$-dimensional integer lattice, while at each 
time unit, they split into independent copies according 
to time-space i.i.d. offspring distributions. 
As noted by Comets and Yoshida, the BRWRE is naturally associated with the 
directed polymers in random environment (DPRE), 
for which the quantity $\Psi$ called the free energy is well studied. 
Comets and Yoshida proved that there is no survival when $\Psi<0$ and that
survival is possible when $\Psi>0$.
We proved here that, except for degenerate cases, the BRWRE always die when $\Psi=0$. This solves a conjecture of Comets and Yoshida.

\end{abstract}

{\maketitle
}
\setcounter{tocdepth}{1}

\section{Introduction}

\subsection{Branching random walks in random environment}

Let us introduce the model. We write $\N=\{0,1,2,\dots\}$ and $\N^*=\{1,2,\dots\}$.

To each $(t,x)\in \N \times \Zd$, we associate a distribution
$q_{t,x}=(q_{t,x}(k))_{k \in \N}$ on the integers; the family $\mathbf{q}=(q_{t,x})_{t\in\N,x\in\Zd}$ is called an environment. 
We denote by $\Lambda=\mathcal{P} (\N)^{\N \times \Zd}$ the space of environments, where $\mathcal P(\N)$
is the set of distributions on $\N$.

Given an environment $\mathbf{q}=(q_{t,x})_{t\in\N,x\in\Zd}\in\Lambda$, we define the branching random walk in environment ${\bf q}$ as the following dynamics:
\begin{itemize} 
\item 
At time $t=0$, there is one particle 
at the origin $x=0$. 
\item 
Each particle, located at
site $x \in \Zd$ at time $t$, jumps at time $t+1$ to one of the $2d$ neighbors 
of $x$ chosen uniformly; upon  arrival, it dies and is replaced by $k$ new particles 
with probability 
$q_{t,x}(k)$. The number of newborn particles is independent of the jump,
and all these variables, indexed by the full population at time $t$, 
are independent. 
\end{itemize}
Take a random  $\mathbf{q}$, we obtain a branching random walk in random environment (BRWRE).
We assume here that $\mathbf{q}=(q_{t,x})_{t\in\N,x\in\Zd}$ is a $\mathcal{P}(\N)$-valued  i.i.d. sequence with common distribution $\gamma$ and 
we denote 
by $\P^\gamma$ the annealed law of the branching random walk in this random environment.
We also note $\E^\gamma$ for the expectation with respect to $\P^\gamma$.

Since the environment not only depends on the sites but also on the time, we can be more specific and say that we work  with a random space-time or random dynamic environment.
This is the  model studied in Yoshida~\cite{MR2434183}, Hu and Yoshida~\cite{MR2513122}, and also Comets and Yoshida~\cite{MR2822477}.
A natural question is to characterize the laws $\gamma$ that allow the branching random walk to survive.

For $(t,x) \in \N\times \Zd$ and fixed $\mathbf{q}$, we introduce
the mean progeny at site $(t,x)$:
\begin{equation}
\label{Defm}
m_{t,x}=\sum_{k \in \N}kq_{t,x}(k).
\end{equation}
 From here on, we assume that 
\begin{align}
& \P^{\gamma}(m_{0,0}+m_{0,0}^{-1}) < +\infty, \label{hyp1} \\
& \P^{\gamma}(q_{0,0}(0)>0)>0 \quad \text{and} \quad \P^{\gamma}(q_{0,0}(0)+q_{0,0}(1)<1)>0. \label{hyp2} 
\end{align}
Assumptions \eqref{hyp2} are intended to avoid obvious survival and obvious extinction.

\subsection{The Comets--Yoshida Theorem and the associated conjecture}
The Comets--Yoshida Theorem~\cite{MR2822477} relates  the survival of the BRWRE with a functional on an associated directed polymer in random environment as follows: 
define on a probability space $(\Omega_S, \mathcal{F}_S, P_S)$ a  simple symmetric random walk $(S_t)_{t \ge 0}$ on 
$\Zd$ starting from $S_0=0$.
The partition function of the directed polymer in random environment ${\bf q}$ is given by
\begin{equation}\label{z_t}
Z_t=\int \prod_{u=0}^{t-1}m_{u,S_u}\quad dP_S,
\end{equation}
with $m_{t,x}$ as in \eqref{Defm}. 
It is easy to see (e.g., \cite[Lemma 1.4]{MR2434183}) that
$Z_t$ is the expectation of the number of particles of the BRWRE living at time $t$, knowing the random environment ${\bf q}=(q_{t,x})_{(t,x)
\in \N \times \Zd}$. 
Note that \eqref{hyp1}, combined with the inequality $|\log u| \le u \vee u^{-1}$ for $u >0$,
 implies that
\begin{equation}\label{lnZL^1}
\forall t\in\N^*\quad \E^{\gamma}( |\log Z_t|) <+\infty  
\end{equation}
 \begin{propo}[Comets--Yoshida]
\label{Comets-Yoshida}
The following limit exists
\begin{equation}\label{psi(0)} 
\Psi(\gamma) \stackrel{\rm def.}{=}\lim_{t \to \infty}\frac{1}{t} \E^{\gamma}(\log Z_t).
\end{equation}
Moreover,
\begin{itemize}
\item If $\Psi(\gamma)<0$, then  $\P^{\gamma}(\mathrm{survival})=0$.
\item If $\Psi(\gamma)>0$ and $\E^{\gamma}[-\log({1-q_{0,0}(0)})]<+\infty$, then  $\P^{\gamma}(\mathrm{survival}\ |{\bf q})>0$ for $\P^{\gamma}$-almost every ${\bf q}$.
\end{itemize}
\end{propo}
The number $\Psi$ is called the free energy of the polymer.
The aim of this article is to prove the following theorem, stated as a conjecture by Comets and Yoshida~\cite{MR2822477}:

\begin{theorem}
\label{letheoreme}
Assume that~\eqref{hyp1} and~\eqref{hyp2} are satisfied.
Then
$$(\Psi(\gamma)=0) \Longrightarrow (\P^{\gamma}(\mathrm{survival})=0).$$
\end{theorem}
It is a well-known fact that growth processes often die out in the critical case: it obviously holds for the Galton--Watson tree, but it also holds for more complicated models.
This has been proved for the contact process in dimension $d$ by Bezuidenhout and Grimmett in a famous paper entitled \emph{The critical contact process dies out}~\cite{MR1071804}. Recently, Steif and Warfheimer~\cite{MR2461788}  generalized this result to the case of the contact process in randomly evolving environment introduced by Broman~\cite{MR2353388}. We also proved in \cite{bacteries} that is it true for a class of dependent oriented percolations on $\Zd$.

\subsection{Strategy of the proof}

Recall that $\gamma$ is the common law of the i.i.d sequence $(q_{t,x})_{t \ge 0, x \in \Zd}$ of distributions on $\N$. For  $\rho\in (0,1]$, we define a perturbation $\gamma^{\rho}$  of $\gamma$  by setting:  $$\forall k\ge 0\quad \gamma^\rho(k)=\rho\gamma(k)+(1-\rho){\delta_0}(k),$$ where ${\delta_0}$ is the Dirac mass  on $0$. Under $\gamma^\rho$, the environment is then less favorable to births of particles than under $\gamma$. 
\begin{lemme}
\label{pertu}
$\forall \rho\in (0,1]\quad \forall \gamma\in (0,1]\quad \Psi(\gamma^\rho)=\log \rho+\Psi(\gamma).$
\end{lemme}

\begin{proof}
Note $\Q^{\gamma}=\gamma^{\otimes \N\times\Zd}$.
Let $q^{\rho}_{t,x}(k)=\rho q_{t,x}(k) +(1-\rho)\delta_0(k)$.
Obviously, the law of the field $(q^\rho_{t,x})$ under $\Q^{\gamma}$ is $\Q^{\gamma^\rho}$.
If we define, for $(t,x) \in \N\times\Zd$,
$$ m_{t,x}^{\rho}=\sum_{k \in \N}kq^{\rho}_{t,x}(k) \quad \text{and} \quad Z^\rho_t=\int \prod_{u=0}^{t-1}m^\rho_{u,S_u}\quad dP_S,$$
it is easy to see that $m^{\rho}_{t,x}=\rho m_{t,x}$, that $Z^{\rho}_t=\rho^t Z_t$, and finally
$$\Psi(\gamma^\rho)=\lim_{t \to \infty}\frac{1}{t} \E^{\gamma}[\log Z^\rho_t]=\log \rho+\lim_{t \to \infty}\frac{1}{t} \E^{\gamma}[\log Z_t]=\log \rho+\Psi(\gamma).$$
\end{proof}

To prove Theorem~\ref{letheoreme}, we will use the strategy of Bezuidenhout and Grimmett, which consists in characterising survival by a local event. More precisely, for a fixed law $\gamma$ allowing the branching random walk to survive, we will exhibit a local event $A$, that only depends on what happens in a finite box $[0,T]\times [-L,L]^d$, and a level $p_0<1$ such that
\begin{itemize}
\item[$(C_1)$] $\P^{\gamma}(A)>p_0$
\item[$(C_2)$] For every law $\gamma'$ on $\mathcal{P}(\N)$, $(\P^{\gamma'}(A)>p_0) \Longrightarrow (\P^{\gamma'}(\text{survival})>0)$.
\end{itemize}
A simple coupling argument implies that
$\P^{\gamma^\rho}(A)\ge \rho^{(T+1)(2L+1)^d}\P^{\gamma}(A),$
and thus
\begin{equation}
\label{pertu2}
(\P^{\gamma}(\text{survival})>0) \Longrightarrow (\exists \rho\in (0,1)\quad \P^{\gamma^{\rho}}(\text{survival})>0).
\end{equation}
Consider now $\gamma$ such that $\P^{\gamma}(\text{survival})>0$, take $\rho$ given by \eqref{pertu2} and assume $\Psi(\gamma)=0$. Lemma~\ref{pertu} ensures $\Psi(\gamma^\rho)<0$ while $\P^{\gamma^{\rho}}(\text{survival})>0$, which is in contradiction with Proposition~\ref{Comets-Yoshida}. Thus
$$(\P^{\gamma}(\text{survival})>0) \Longrightarrow (\Psi(\gamma)>0),$$
which proves Theorem \ref{letheoreme}. 

\medskip
We focus now on the Bezuidenhout--Grimmett construction for the BRWRE.
First we need to give a more precise construction of the BRWRE. In particular, we will need a FKG inequality, which we prove in the next section.

\section{Notations and FKG inequality}

An initial configuration $A=(A(x))_{x \in \Zd}$ is an element of the set $\N^{\Zd}$ satisfying 
$$|A|=\sum_{x \in \Zd} A(x)<+\infty.$$
The number $A(x)$ is the number of particles sitting on site $x$ at time $t=0$. We encode the BRWRE starting from the initial configuration $A$ by the random variables $\eta^A_t=(\eta^A_t(x))_{t\in\N, x\in\Zd}$ which represent the number of particles living on site $x$ at time $t$. Thus, 
$$|\eta^A_t|=\sum_{x \in \Zd} \eta^A_t(x)$$ stands for the total number of particles living at time $t$. To exploit the independence properties of the environment, we also need to consider the occupied sites at time $t$:
$$\underline{\eta}^A_t(x)= \mathbf{1}_{\{\eta^A_t(x)>0\}} \quad \text{and} \quad \underline{\eta}^A_t= \{x \in \Zd: \; \eta^A_t(x)>0\}.$$
We denote by $|\underline{\eta}^A_t|$ the number of occupied sites at time $t$. We define the lifetime of the branching random walk starting from $A$
$$\tau^A = \min\{t \in \N: \; \eta^A_t=0_{\Zd}\}=\min\{t \in \N: \; \underline{\eta}^A_t=\varnothing\}.$$
We say that there is survival starting from $A$ when $\tau^A=+\infty$.

In this article, the genealogy of the BRWRE is not central, and we choose
 a description in terms of particles systems. In the proof of the next lemma on FKG inequalities, we give a full construction of the BRWRE in the spirit of probabilistic cellular automata. 

\begin{lemme}
\label{LemFKG} For any non decreasing functions $f,g:\N^{\N\times\Zd} \to \R_+$,
\begin{align}
\E\left[ f\left(\eta_{t}^A \right) g\left(\eta_{t}^A\right) \right] &\ge \E\left[ f\left(\eta_{t}^A\right)\right] \E\left[g\left(\eta_{t}^A\right) \right], \\
\E\left[ f\left(\underline{\eta}_{t}^A \right) g\left(\underline{\eta}_{t}^A\right) \right] &\ge \E\left[ f\left(\underline{\eta}_{t}^A\right)\right] \E\left[g\left(\underline{\eta}_{t}^A\right) \right].
\end{align} 
\end{lemme}

\begin{proof}
To build  the BRWRE, we need  a probability space on which live the following random variables:
\begin{itemize}
  \item The $(q_{t,x})_{(t,x)\in \N\times \Zd}$ are i.i.d, take their values in $\mathcal{P}(\N)$, and have $\gamma$ as common law.
  \item The $(U_{t,x,k})_{(t,x,k)\in \N\times \Zd\times \N^*}$ are i.i.d and follow the uniform law on $[0,1]$. With the environment $\mathbf{q}$, the random variable  $U_{t,x,k}$ is used to generate the progeny of the $k$-th particle sitting on site $x$ at time $t$. 
  \item The $(D_{t,x,k})_{(t,x,k)\in \N\times \Zd\times \N^*}$ are i.i.d and follow the uniform law on $V=\{v\in\Zd:\;\|v\|_1=1\}$. The random variable $D_{t,x,k}$ gives the displacement of the  children of the $k$-th particle sitting on site $x$ at time $t$. 
\item All these variables are independent.
\end{itemize} 
So we can choose $\Omega=\mathcal{P}(\N)^{\N\times\Zd}\times [0,1]^{\N\times \Zd\times \N^*}\times V^{\N\times\Zd\times\N^*}$ and 
$$\P^{\gamma}={\gamma}^{\otimes \N\times\Zd}\otimes \mathcal U([0,1])^{\otimes \N\times \Zd\times \N^*}\otimes \mathcal U(V)^{\otimes \N\times\Zd\times\N^*}.$$
For $t\ge 1$, we denote by $\theta_t$ the time translation operator acting on each coordinate of $\Omega$.

We define then the number $E_{t,x}(v,p)$ of children born on site $x$ at time $t$ and moving to $x+v$ when $p$ particles live on $x$ at time $t$:
$$E_{t,x}(v,p)=\sum_{k=1}^p\1_{\{D_{t,x,k}=v\}} \sum_{i=1}^{+\infty} \1_{\{\sum_{j=0}^i q_{t,x}(j)\le U_{t,x,k}\}}.$$
An initial configuration $A \in \N^{\Zd}$ being given, we set $\eta_0=A$, and
$$\eta_{t+1}^A(x)=\sum_{v \in V} E_{t,x-v}(v, \eta_{t}^A(x-e)).$$
We note $\mathcal{F}_t=\sigma((q_{s,x},U_{s,x,k},D_{s,x,k})_{0\le s\le t,x\in\Zd,k\in \N^*})$. Then, $\eta_{t+1}$ is  $\mathcal{F}_t$-measurable.

Note that
\begin{itemize}
  \item for every $\omega\in \Omega$ and $v \in V$,  $p\mapsto E_{t,x}(v,p)(\omega)$ is non-decreasing.
  \item The coordinates of the random field $(E_{t,x})_{t\in \N,x\in\Zd}$ are i.i.d. under $\P^\gamma$.
\end{itemize}
By induction on $t \in \N$, we thus see that $A \mapsto \eta_{t}^A(x)$ is non-decreasing; also for fixed $A$, the $(\eta_{t}^A(x))_{t\in \N,x\in\Zd}$ are non-decreasing with respect to the vectors $(E_{t,x})_{t\in \N,x\in\Zd}$. As these vectors are independent, Holley's inequality holds for the $(\eta_{t}^A(x))_{t\in \N,x\in\Zd}$. 

As the $(\underline{\eta}^A_t(x))_{t\in \N,x\in\Zd}$ are non-decreasing functions of the $(\eta_{t}^A(x))_{t\in \N,x\in\Zd}$, Holley's inequality also holds for the $(\underline{\eta}^A_t(x))_{t\in \N,x\in\Zd}$.
\end{proof}

\subsection*{Truncating} To localize the events we consider, we also need truncated versions of the BRWRE: for $B\subset\Zd$, and $A \in \N^{B}$ such that $|A|<+\infty$, we set ${}_{B}\eta_0^A=A$ and
$${}_{B}\eta_{t+1}^A(x)=\1_{B}(x) \sum_{v\in V} E_{t,x-v}(v, {}_{B}\eta_{t}^A(x-v)).$$
In words, we keep in ${}_{B}\eta_t^A$ only particles living in $B$: other particles are simply discarded. Note that ${}_{B}\eta_{t}^A(x)$ is obviously non-decreasing in $L$ and non-decreasing in $B$. A straight adaptation of the previous proof ensures that the FKG inequalities of Lemma \ref{LemFKG} still hold for the truncated BRWRE.
Note that the process $({}_{B}\eta_{t}^A)_{t\ge 0}$ is $\sigma((q_{t,x},U_{t,x,k},D_{t,x,k})_{t\in\N,x\in B,k\in \N^*})$-measurable.

We mainly work with restrictions on subsets $B=\{-L,\dots,L\}^d$. In that case, we simply note
${}_{L}\eta_{t}^A$ instead of  ${}_{\{-L,\dots,L\}^d}\eta_{t}^A$.

\section{Construction of the block event}

\subsection*{Outline of proof}

Recall that we are looking for a local event $A$ which satisfies $(C_1)$ and $(C_2)$. The idea is to find an event that expresses the fact that if the branching random walk occupies a sufficiently large area at a given place, it will presumably extend itself a bit further.

 Set, for every integer $n \ge 1$, 
$$A_n=\left\{x \in \mathbb Z^d: \; \|x\|_1 \le n \text{ and } \sum_{i=1}^n x_i\equiv 0 \quad [2] \right\}.$$
Note that when $n$ is even, $A_n$ is the maximal set that a branching random walk can reach at time~$n$.

The next proposition ensures that starting from $A_n$ with $n$ large enough, the BRWRE restrained to a large box $\{-2L-2n,\dots,2L+2n\}^d\times \{0,\dots,2T\}$ occupies with high probability a translated copy of $A_n$:

\begin{propo}
\label{unpas}
For every $\varepsilon>0$, there exist positive integers $n,L,T$, with $n \le L$ such that
$$\P^{\gamma} \left( 
\begin{array}{c}
\exists x \in \{L+n,\dots,2L+n\}\times \{0,\dots ,2L\}^{d-1}, t \in \{T,\dots, 2T\} \\
 {}_{2L+2n}{\underline{\eta}}_{t}^{A_n} \supset x+A_n 
 \end{array}\right)\ge 1-\varepsilon.$$
\end{propo}

Proposition~\ref{unpas} provides an event $A$ that satisfies $(C_1)$.
The fact that $A$ also satisfies $(C_2)$ follows from a quite standard construction of an embedded supercritical percolation of blocks.
This construction does not rely on the specificities of the model.
We can find in the literature many examples of similar block events. Most of these papers adapt the initial construction of the Bezuidenhout--Grimmett article~\cite{MR1071804}: \emph{The critical contact process dies out}. Their proof is also exposed in the reference book by Liggett~\cite{MR1717346}.

A complete description of these procedures can be found in  Steif--Warfheimer~\cite{MR2461788} in the case of a contact process where the death rate depends on a dynamical environment or in Garet--Marchand~\cite{bacteries} for a class of dependent oriented percolation.

\subsection{Some properties of the surviving branching random walk}
\label{tropdur}
We fix $\gamma$ such that the probability of survival $\P^\gamma(\tau^0=+\infty)$ is positive.
At first, we prove that survival implies the explosion of the number of particles.

\begin{lemme}
\label{caexplose1}  
For every finite initial configuration $A$,
$$\P^{\gamma}(\tau^A=+\infty, \; \lim_{t \to +\infty} |\eta_{t}^{A}|=+\infty)=\P^{\gamma}(\tau^A=+\infty).$$
\end{lemme}

\begin{proof} Let $A$ be a fixed finite initial configuration, and $N$ be a fixed positive integer. By assumption \eqref{hyp2}, there exist $\varepsilon_0, \alpha_0>0$ such that
$$\P^\gamma(q_{0,0}(0)>\varepsilon_0)>\alpha_0.$$
By blocking the progenies of all living particles at time $s$, we see that, 
\begin{align*}
\forall s>1 \quad \P^\gamma \left( \tau^{A}<+\infty \; | \; \mathcal{F}_{s-1} \right)
& \ge  \varepsilon_0^{|\eta^A_s|} \alpha_0^{|\underline{\eta}^A_s|} \ge (\varepsilon_0 \alpha_0)^{|\eta^A_s|}.
\end{align*}
By the martingale convergence theorem, $\displaystyle \lim_{s \to + \infty} \P^{\gamma} \left( \tau^{A}<+\infty \; | \; \mathcal{F}_{s-1} \right)=\1_{\{\tau^{A}<+\infty\}}$. So on the event $\{\tau^{A}=+\infty\}$, $\displaystyle \lim_{s \to + \infty} |\eta_{s}^{A}|=+\infty$. 
\end{proof}

To exploit the independence properties of the environment, we need particles to sit on distinct sites; the fact that the number of particles at time $t$ explodes when the BRWRE survives does not directly ensure that the number of occupied sites also explodes. It will be a by-product of our construction, but we are not able to prove it at this stage. So we need to work separately under the two complementary assumptions:
\begin{align*}
 (H)&  \quad \P^\gamma (q_{0,0}(0)=1)> 0, \\
  (H^c)& \quad \P^\gamma (q_{0,0}(\N^*)>0)=1.
\end{align*}

Under Assumption~$(H)$, we prove that survival implies the explosion of the number of occupied sites:

\begin{lemme}
\label{caexplose2} 
Assume $(H)$ is fulfilled. 
For every finite initial configuration $A$, 
$$\P^\gamma(\tau^A=+\infty, \; \lim_{t \to +\infty} |\underline{\eta}_{t}^{A}|=+\infty)=\P^{\gamma}(\tau^A=+\infty).$$
\end{lemme}

\begin{proof} Let $A$ be a fixed finite initial configuration, and $N$ be a fixed positive integer. Under $(H)$
$$\exists \varepsilon_1>0 \quad \P^\gamma (q_{0,0}(0)=1) \ge \varepsilon_1.$$
By blocking the progenies of all occupied sites at time $s$, we see that, 
\begin{align*}
\forall s \ge 1 \quad \P^\gamma \left( \tau^{A}<+\infty \; | \; \mathcal{F}_{s-1} \right)
& \ge  \varepsilon_1^{|\underline{\eta}^A_s|} .
\end{align*}
By the martingale convergence theorem, $\displaystyle \lim_{s \to + \infty} \P^{\gamma} \left( \tau^{A}<+\infty \; | \; \mathcal{F}_{s-1} \right)=\1_{\{\tau^{A}<+\infty\}}$. So on the event $\{\tau^{A}=+\infty\}$, $\displaystyle \lim_{s \to + \infty} |\underline{\eta}^A_s|=+\infty$. 
\end{proof}

Besides, under Assumption~$(H^c)$, we prove that if we start with many particles sitting on the same site, there is a large probability that after some time many sites are occupied.
We denote by $(N)$ the initial configuration where $N$ particles sit on $0$.
Formally, we have $(N)=(N\delta_0(x))_{x\in\Zd}$.

\begin{lemme}
\label{casetend}
Assume $(H^c)$ is fullfilled. Let $n$ be a fixed even integer. 
Set $B_n=\{0,\dots,2n\}\times\{-n,\dots,n\}^{d-1}$.
Then
$\displaystyle \lim_{N \to +\infty} \P^\gamma(ne_1+A_n \subset {}_{B_n}\eta^{(N)}_{2n})=1.$
\end{lemme}

\begin{proof} Fix $n \ge1$ and $\alpha>0$. Under $(H^c)$, 
\begin{equation} 
\label{epsiloneta}
\forall \varepsilon >0 \quad \exists \eta>0 \quad \P^\gamma (q_{0,0}(\N \backslash \{0\})\ge \eta) \ge 1-\varepsilon/2.
\end{equation}
We say that the environment $q$ is fertile in the box $B_n\times \{0,\dots,2n\}$ if 
$$\forall x \in B_n \quad \forall t \in \{0,\dots,2n\} \quad q_{x,t}(\N \backslash \{0\}) \ge \eta;$$
we denote this event by $F$. Thus, with $\varepsilon$ and $\eta$ satisfying \eqref{epsiloneta},
$$\P^\gamma(F) \ge (1-\varepsilon)^{(4n+1)^d(2n+1)}.$$
We choose $\varepsilon>0$ such that $\P^\gamma(F) \ge 1-\alpha/2$, and take the corresponding $\eta$ given by~\eqref{epsiloneta}. Now, as $\P^\gamma(ne_1+A_n \subset \eta^{(N)}_{2n}) \ge \P^\gamma(ne_1+A_n \subset \eta^{(N)}_{2n}|F)(1-\alpha/2)$, we restrict ourselves to environments in $F$. For an environment $q$ in $F$, each coordinate $q_{x,t}$, $x \in B_n$, $t \in \{0,\dots,2n\}$  stochastically dominates the fixed law $\gamma_0=\eta \delta_1+(1-\eta) \delta_0$. Thus by coupling, it is sufficient to prove that 
\begin{equation}
\label{limiteN}
\lim_{N \to +\infty} \P^{\gamma_0}(ne_1+A_n \subset {}_{B_n}\eta^{(N)}_{2n})=1.
\end{equation}
Let now $x \in ne_1+A_n$ be fixed. As $\|x\|_1\le 2n$,
$$\P^{\gamma_0}(x \in  {}_{B_n}\eta^{(1)}_{2n}) \ge \left( \frac{\eta}{2d} \right)^{2n}>0.$$
 As the environment is non-random, the behavior of distinct particles are independent under $\P^{\gamma_0}$. Thus, 
 $$\lim_{N \to +\infty} \P^{\gamma_0}(x \in {}_{B_n}\eta^{(N)}_{2n})\ge \lim_{N \to +\infty} 1-\left( 1-\left( \frac{\eta}{2d} \right)^{2n}\right)^N=1,$$
 which proves \eqref{limiteN}; this ends the proof.
\end{proof}

A last lemma ensures that starting from $A_{n}$ with $n$ large, the survival probability  is large:
\begin{lemme}
\label{lem00} 
$\displaystyle \lim_{n \to + \infty} \P^\gamma \left (\tau^{A_{n}}=+\infty \right)=1$.
\end{lemme}

\begin{proof}
By monotonicity, $\displaystyle \lim_{n \to + \infty} \P^\gamma \left (\tau^{A_{2n}}=+\infty \right)=\P^\gamma \left (\tau^{A_\infty}=+\infty \right)$, with $$A_\infty=\{x \in \mathbb Z^d: \;  \sum_{i=1}^n x_i \in 2\Z\}.$$
The event $\{\tau^{A_\infty}=+\infty\}$ is invariant under the spatial translation $x\mapsto x+2e_1$, so by ergodicity, it is either null or full. Since  $\P^{\gamma}(\tau^{A_\infty} =+\infty)\ge \P^\gamma (\tau^{0} =+\infty)>0$, it is equal to $1$.
\end{proof}





\begin{lemme}
\label{lem0} 
Recall that
$B_n=\{0,\dots,2n\}\times\{-n,\dots,n\}^{d-1}$.
For every $n\ge1$, there exists $h>0$ such that 
$$\P^\gamma(A_n \subset \; _n\underline{\eta}^0_{h})>0 \text{ and } \P^\gamma(ne_1+A_n \subset \; _{B_n}\underline{\eta}^0_{h})>0.$$
\end{lemme}

\begin{proof}
By~\eqref{hyp2}, there exists $\eta>0$ and $\epsilon>0$ such that $$\P^{\gamma}(q_{0,0}(0)+q_{0,0}(1)\le 1-\eta)\ge \epsilon.$$
Let $k$ be even with $2^k\ge |A_n|$ and let $h=k+2n$.
Note $$G=\left\{\forall (x,t)\in B_n\times \{0,\dots,k+2n\}\quad q_{0,0}(0)+q_{0,0}(1)\le 1-\eta\right\}.$$
$\P^{\gamma}(G)\ge \epsilon^{|B_n\times \{0,\dots,h\}|}>0$.
On the event $G$, the environment allows each individual to have more than one daughter with probability $\eta$.
For an environment $q$ in $F$, each coordinate $q_{x,t}$, $x \in B_n$, $t \in \{0,\dots,h\}$  stochastically dominates the fixed law $\gamma_1=\eta\delta_2+(1-\eta)\delta_0$, so 
for each event $E$, we have $$\P^{\gamma}(G)\ge\epsilon^{|B_n\times \{0,\dots,h\}|} \P^{\gamma_1}(G).$$
Thus, it is sufficient to prove that $\P^{\gamma_1}(A_n \subset \; _n\underline{\eta}^0_{h})>0$ and  $\P^{\gamma_1}(ne_1+A_n \subset \; _{B_n}\underline{\eta}^0_{h})>0$.

We only prove the second inequality; the first one is similar.

Under $\P^{\gamma_1}$, the probability that there is at least $|A_n|$ particles at the origin at time $k$ whose ancesters never left $\{0;e_1\}$ is at least $(\frac{\eta}{2d})^{2^{k+1}-1}$.
Then, there is a probability at least $(\frac{\eta}{2d})^{2n|A_n|}$ that their progeny fills $ne_1+A_n$ in $2n$ steps, with trajectories remaining inside $\{-2n,\dots 2n\}^d$.
Finally, $$\P^{\gamma_1}(ne_1+A_n \subset \; _{B_n}\underline{\eta}^0_{h})\ge (\frac{\eta}{2d})^{2^{k+1}-1}(\frac{\eta}{2d})^{2n|A_n|}>0.$$
The other proof is similar.
\end{proof}

\subsection{From survival to local events}
 
\begin{lemme}
\label{lem2}
For every finite $A \subset \Zd$, for every positive integer $N$,
$$\lim_{t \to + \infty} \lim_{L \to + \infty} \P \left (|_{L}\eta_{t}^{A}| \ge N\right)= \P \left (\tau^A=+\infty \right).$$
Under the extra assumption $(H)$, 
$$\lim_{t \to + \infty} \lim_{L \to + \infty} \P \left (|_{L}\underline{\eta}_{t}^{A}| \ge N\right)= \P \left ( \tau^A=+\infty \right).$$
\end{lemme}

\begin{proof} Let $A$ be a fixed finite subset of $\Zd$, and $N$ be a fixed positive integer.  
$$\forall t \in \N \quad \eta_{t}^{A} = 
\bigcup_{L \in \N}   {}_{L}\eta_{t}^{A} .$$
Indeed, the inclusion $\supset$ follows from positivity, and if $L \ge \|A\|_\infty+2t+1$, then for every $s \le t$, 
$\eta_{s}^{A}=\;_{L}\eta_{s}^{A}$. 
Thus
$$\left\{ |\eta_{t}^{A}| \ge N\right\} =
\bigcup_{L \in \N} \left\{ |_{L}\eta_{t}^{A}| \ge N\right\}$$
and thus $\displaystyle \lim_{L \to + \infty} \P \left(|_{L}\eta_{t}^{A}| \ge N\right)= \P \left(|\eta_{t}^{A}| \ge N\right)$. 
But Lemma~\ref{caexplose1} and the dominated convergence theorem ensure that 
$$ \lim_{t \to + \infty} \P^\gamma(\tau^A=+\infty, \; |\eta_{t}^{A}| \ge N) =\P( \tau^{A}=+\infty).$$
In the same way, 
$\displaystyle \lim_{L \to + \infty} \P \left(|_{L}\underline{\eta}_{t}^{A}| \ge N\right)= \P \left(|\underline{\eta}_{t}^{A}| \ge N\right)$.
Under Assumption $(H)$, we use Lemma~\ref{caexplose2} to conclude.
\end{proof}

Then, using the FKG inequality with a classical square root trick, we can ensure that the truncated process at time $t$ contains many points in a prescribed orthant of $\Zd$:

\begin{lemme}
\label{FKG1}
For every positive integers $n,N,t$, for every integer $L >n$,
\begin{align*}
\P^{\gamma} \left (|_{L}\eta_{t}^{A_n}\cap \{0,\dots,L\}^d| \le N\right)^{2^d} 
& \le  \P^{\gamma} \left (|_{L}\eta_{t}^{A_n} | \le N2^d\right), \\
\P^{\gamma} \left (|_{L}\underline{\eta}_{t}^{A_n}\cap \{0,\dots,L\}^d| \le N\right)^{2^d} 
& \le  \P^{\gamma} \left (|_{L}\underline{\eta}_{t}^{A_n} | \le N2^d\right).
\end{align*}
\end{lemme}
\begin{proof}
By the symmetries of the model and of $A_n$, the law of the intersection of the process with  a given orthant does not depend on the orthant. With the FKG inequality, we get
$$\P^{\gamma} \left (|_{L}\eta_{t}^{A_n}\cap \{0,\dots,L\}^d| \le N\right)^{2^d} 
 \le \P^{\gamma} \left (\miniop{}{\cap}{\epsilon\in \{+1,-1\}^d}\left\{\left|{}_{L}\eta_{t}^{A_n}\cap \miniop{n}{\prod}{i=1} \epsilon_i\{0,\dots,L\} \right| \le N \right\}\right),$$
which is smaller than $\P^{\gamma} \left (|_{L}\eta_{t}^{A_n} | \le N2^d\right)$. The other proof is similar.
\end{proof}
For every positive integers $n,L,T$ with $L>n$, we define $N^{A_n}(L,T)$ as the number of particles sitting on sites  $(x,t)$ such that 
$t \in \{0,\dots,T\}$, $\|x\|_\infty=L$, $x \in {}_{L}\underline{\eta}_{t}^{A_n}$: this is the number of particles sitting on a given lateral face of the large box. We also define $\underline{N}^{A_n}(L,T)$ as the number of occupied sites $(x,t)$ such that 
$t \in \{0,\dots,T\}$, $\|x\|_\infty=L$, $x \in {}_{L}\underline{\eta}_{t}^{A_n}$: this is the number of occupied sites on a given lateral face of the large box. 
The next point is to ensure that when the BRWRE survives, it must occupy many points and on the top face and on the lateral faces of a large box:

\begin{lemme}
\label{lem3}
For any positive integers $M,N,n$,  for any increasing sequences of integers $(T_j)_{j\ge 1}$ and $(L_j)_{j\ge 1}$,
$$\miniop{}{\lim}{j\to +\infty} \P^{\gamma} \left( 
\begin{array}{c}
N^{A_n}(L_j,T_j,) \\
\le M 
\end{array}\right) \P^{\gamma} \left( 
\begin{array}{c}
|_{L_j}{\eta}_{T_j}^{A_n}| \\
\le N
\end{array}
\right) \le \P^{\gamma} \left( \tau^{A_n}<+\infty \right).$$
Under the extra assumption $(H)$, 
$$\miniop{}{\lim}{j\to +\infty} \P^{\gamma} \left( 
\begin{array}{c}
\underline{N}^{A_n}(L_j,T_j,) \\
\le M 
\end{array}\right) \P^{\gamma} \left( 
\begin{array}{c}
|_{L_j}\underline{\eta}_{T_j}^{A_n}| \\
\le N
\end{array}
\right) \le \P^{\gamma} \left( \tau^{A_n}<+\infty \right).$$
\end{lemme}

\begin{proof}
For integers $L,T$, let $\mathcal F_{L,T}=\sigma((q_{s,x},U_{s,x,k},D_{s,x,k})_{0\le s\le T,x\in\{-L,\dots,L\}^d,k\in \N^*})$. Set $k=M+N$. Let $(T_j)_j$ and $(L_j)_j$ be two increasing sequences of integers.
$$H_j=\{N^{A_n}(L_j,T_j)+|_{L_j}{\eta}_{T_j}^{A_n}| \le k\} \quad \text{ and } \quad G=\{\tau^{A_n}<+\infty\}.$$
We proceed as in Lemma~\ref{caexplose1}. Since $H_j \in \mathcal F_{L_j,T_j-1}$, 
we have 
$$\P^{\gamma}(G|\mathcal F_{L_j,T_j}) \ge \1_{H_j}(\varepsilon_0 \alpha_0)^{k}.$$
By the martingale convergence theorem, $\P^{\gamma}(G|\mathcal F_{L_j,T_j})$ almost surely converges to~$\1_{G}$, which implies that
$$\miniop{}{\limsup}{j\to +\infty} H_j \subset G, \; \text{ and thus } \miniop{}{\limsup}{j\to +\infty}\P^{\gamma}(H_j) \le \P\left( \miniop{}{\limsup}{j\to +\infty} H_j\right) \le \P^{\gamma}(G).$$
Using the FKG inequality once again, note that
\begin{align*}
\P^{\gamma}(H_j) & =  \P^{\gamma}(N^{A_n}(L_j,T_j)+|_{L_j}\underline{\eta}_{T_j}^{A_n}|\le M+N) \\
& \ge  \P^{\gamma} \left( N^{A_n}(L_j,T_j) \le M \right) \P^{\gamma} \left( |_{L_j}\underline{\eta}_{T_j}^{A_n}| \le N\right), 
\end{align*} 
which ends the proof.
The proof of the result under $(H)$ is obtained by adapting the proof as Lemma~\ref{caexplose2} is adapted from Lemma~\ref{caexplose1}.
\end{proof}

Exactly as in Lemma~\ref{FKG1}, the FKG inequality and the symmetries of the process allow to control the number of colonized points in a prescribed orthant of a lateral face of the box $\{-L,\dots,L\}^d\times \{0,\dots,T\}$.

For every positive integers $L,T$ and every finite $A \subset \Zd$, we define $N_+^{A}(L,T)$ as the  number of particles sitting on sites $(x,t)$ such that 
$t \in \{0,\dots,T\}$, $x_1=L$, $x_i \ge 0$ for $2 \le i \le d$, $x \in _{L}\underline{\eta}_{t}^{A}$, and $\underline{N}_+^{A}(L,T)$ as the  number of occupied sites $(x,t)$ such that 
$t \in \{0,\dots,T\}$, $x_1=L$, $x_i \ge 0$ for $2 \le i \le d$, $x \in _{L}\underline{\eta}_{t}^{A}$. Then

\begin{lemme}
\label{FKG2}
For every positive integers $n,N,K,t$, for every integer $L \ge n$,
$$\P^{\gamma} \left ( N_+^{A_n}(L,T)\le M\right)^{d2^d} \le 
\P^{\gamma} \left ( N^{A_n}(L,T)\le Md2^d\right).$$
$$\P^{\gamma} \left ( \underline{N}_+^{A_n}(L,T)\le M\right)^{d2^d} \le 
\P^{\gamma} \left ( \underline{N}^{A_n}(L,T)\le Md2^d\right).$$
\end{lemme}
\begin{proof}
This is the same square root trick as in the proof of Lemma~\ref{FKG1}.
\end{proof}

\subsection{Proof of Proposition~\ref{unpas}}

With the previous lemmas in hand, we can now  prove Proposition~\ref{unpas}.

\subsubsection{First case:  assume that $(H)$ holds.} In that case, we can work with the number of occupied sites. 

Let $\varepsilon>0$ be fixed and $0<\delta<1$ to be chosen later. \\
With Lemma~\ref{lem00}, we first choose a positive even integer $n$ such that 
\begin{equation}
\label{eqsurvie}
\P^\gamma\left (\tau^{A_n} =+\infty \right) > 1 -2\delta^2.
\end{equation} 
We take then $h>0$ given by Lemma~\ref{lem0}.
Choose an integer $N'$ such that
\begin{equation}
\label{choixN'}
\left(1- \P^\gamma\left( _n\underline{\eta}_{h}^{0} \supset A_n\right) \right)^{N'} \le \delta,
\end{equation}
 and then $N$ such that every finite subset $A$ of $\Zd$ with cardinal $N$ contains a subset $A'$ of $N'$ points such that
$$\forall x, y \in A' \quad (x\ne y)\Longrightarrow \|x-y\|_\infty \ge 2n+1.$$
Since background random variables in disjoint areas are independent, as a consequence of~\eqref{choixN'}, we get that for every subset $A$ of $\Zd$.
\begin{align}
\label{silestgrandN}
|A|\ge N& \quad \Longrightarrow\quad  \P^{\gamma}(\exists x\in A\quad {}_{n+\|A\|_{\infty}}\underline{\eta}_{h}^{x} \supset x+A_n)\ge 1-\delta
\end{align}
Similarly, there exists $M>0$ such that
\begin{align}
\label{silestgrandM}
|A|\ge M&\quad  \Longrightarrow \quad \P^{\gamma}(\exists x\in A\quad {}_{2n+\|A\|_{\infty}}\underline{\eta}_{h}^{x} \supset x+ne_1+A_n)\ge 1-\delta
\end{align}
Since $1-2\delta<1-2 \delta^2< \mathbb P^\gamma \left (\tau^{A_n}=+\infty \right)$, there exist by Lemma~\ref{lem2} $T_0\in\N$ and a map $L':[T_0,+\infty)\to \N$ such that
$$\forall T\ge T_0\quad \forall L\ge L'(T)\quad \P^\gamma \left( |_{L}\underline{\eta}_{T}^{A_n}| > 2^dN\right)\ge 1-\delta.$$
Define $L_0=L'(T_0)$, then for $i\ge 0$ 
$$T_{i+1}=\inf\left\{T>T_i;\  \P^\gamma \left( |_{L_i}\underline{\eta}_{T}^{A_n}| > 2^dN\right)<1-2\delta\right\}.$$
Note that $T_{i+1}<+\infty$ is finite because with assumption $(H)$ we have $$\displaystyle \lim_{t \to +\infty} \P^\gamma \left( |_{L_i}\underline{\eta}_{t}^{A_n}| > 2^dN\right)=0.$$
Then simply take $L_{i+1}=\max(L'(T_{i+1}),L_i+1)$. Then, we have
two increasing sequences $(T_i)_{i\ge 1}$ and $(L_i)_{i\ge 1}$ such that
\begin{align}
\label{choixT1L1}
&\forall i\ge 1 \quad \P^\gamma \left( |_{L_i}\underline{\eta}_{T_i}^{A_n}| > 2^dN\right) \ge 1 -2 \delta\\
\label{condition limite}
&\forall i \ge 1\quad \P^\gamma \left( |_{L_i}\underline{\eta}_{T_i+1}^{A_n}| > 2^dN\right) < 1 -2 \delta.
\end{align}
With Lemma~\ref{lem3} and \eqref{eqsurvie}, there exists $K$ such that 
$$
\P^\gamma  \left( \begin{array}{c}
\underline{N}^{A_n}(L_K,T_K+1) \\
\le d2^dM 
\end{array}
\right) \P^\gamma \left( \begin{array}{c}
|_{L_K}\underline{\eta}_{T_K+1}^{A_n}| \\ \le 2^dN \end{array} \right)  
\le  2\delta^2.
$$
From now on, set $L=L_K$ and $T=T_K+1$.
We get
\begin{align*}
\P^\gamma \left( |_{L}\underline{\eta}_{T-1}^{A_n}| > 2^dN\right) & \ge  1 -2 \delta, \\
\P^\gamma  \left( \underline{N}^{A_n}(L,T) > d2^dM \right)  
& \ge   1 - \frac{2\delta^2}{1-\P \left( 
|_{L}\underline{\eta}_{T}^{A_n}|  > 2^dN \right) }
\ge 1-\delta.
\end{align*}
With lemmas~\ref{FKG1} and \ref{FKG2}, we obtain
\begin{align}
\label{enhaut} \P^\gamma \left( |_{L}\underline{\eta}_{T-1}^{A_n}\cap \{0,\dots,L\}^d| > N\right) & \ge   1 -(2 \delta)^{2^{-d}}, \\
\label{cote} \P^\gamma  \left( \underline{N}_+^{A_n}(L,T) > M \right)  &\ge  1-\delta^{2^{-d}/d}. 
\end{align}
Putting \eqref{silestgrandN} and \eqref{enhaut} together, we get
\begin{equation}
 \P^\gamma \left( \exists x \in \{0,\dots,L\}^d  \quad _{L+2n}\underline{\eta}_{T-1+h}^{A_n} \supset x+A_n 
\right)  \ge  (1 -(2 \delta)^{2^{-d}})(1-\delta).
\end{equation}
Now, we can use \eqref{cote} and \eqref{silestgrandM} starting from the translated $x+A_n$ occupied at time $T-1+h$: it leads to
\begin{equation}
\P^\gamma \left( \begin{array}{c}
\exists x \in \{L+n,\dots,2L+n\}\times \{0,\dots,2L\}^{d-1}, \\
\exists t \in \{T-1+2h,\dots,2T-1+2h\} \\
 _{2L+2n}\eta_{t}^{A_n} \supset x+A_n
\end{array} \right)  \ge  (1 -(2 \delta)^{2^{-d}})(1-\delta)(1-\delta^{2^{-d}/d})(1-\delta).
\end{equation}

By choosing $\delta>0$ small enough, the right-hand side can be made larger than $1-\varepsilon$, and renaming $T+h$ by $T$ gives the result. 

\subsubsection{Second case: assume that $(H^c)$ holds.} We can not work with the number of occupied sites anymore, so we must work with the number of particles.

Let $\varepsilon>0$.
Let $\delta\in (0,1)$ be such that $\delta \le\epsilon/4$ and $$(1 -(2 \delta)^{2^{-d}})(1-\delta)(1-\delta^{2^{-d}/d})(1-\delta)\ge 1-\epsilon.$$
We must again split the proof into two complementary subcases.

\begin{align}
&(A_{\delta}):\quad \forall L>0\quad  \miniop{}{\limsup}{T \to +\infty}  \P^\gamma (|_{L}\eta_{T}^{A_n}| > 0)< 1-2\delta\\
&(A_{\delta}^c):\quad \exists L>0 \quad  \miniop{}{\limsup}{T \to +\infty}  \P^\gamma (|_{L}\eta_{T}^{A_n}| >0)\ge  1-2\delta.
\end{align}

\underline{Suppose first that  $(H^c)$  and $(A_{\delta})$ hold.}

Choose $n$, $N$, $M$ exactly as in \eqref{eqsurvie}, \eqref{silestgrandN}  and
\eqref{silestgrandM}. Under $(H^c)$, by Lemma~\ref{casetend}, we can choose $J$ large enough to have
\begin{equation}
\label{choixJ}
\left(1- \P^\gamma\left( ne_1+A_n \subset {}_{B_n}\eta^{(J)}_{2n} \right) \right) \le \delta.
\end{equation}
Set $R=\max\{N,M,J\}^2$.

Using  $1-2\delta<1-2 \delta^2< \mathbb P^\gamma \left (\tau^{A_n}=+\infty \right)$, $(A_{\delta})$
and Lemma~\ref{lem2} together, the same reasoning as in case $(H)$ gives the existence of two increasing sequences $(T_i)_{i\ge 1}$ and $(L_i)_{i\ge 1}$ such that
\begin{align}
\label{choixT1L1bis}
&\forall i\ge 1 \quad \P^\gamma \left( |_{L_i}\eta_{T_i}^{A_n}| > 2^dR\right) \ge 1 -2 \delta\\
\label{conditionlimite}
&\forall i\ge 1 \quad \P^\gamma \left( |_{L_i}\eta_{T_i+1}^{A_n}| > 2^dR\right) < 1 -2 \delta.
\end{align}

By Lemma~\ref{lem3} and \eqref{eqsurvie}, there exists $K$ such that 
$$
\P^\gamma  \left( N^{A_n}(L_K,T_K+1) \le d2^dR \right) 
\P^\gamma \left(|_{L_K}\eta_{T_K+1}^{A_n}|  \le 2^dR \right)  
\le  2\delta^2,
$$
From now on, set $L=L_K$ and $T=T_K+1$.
We get
\begin{align*}
\P^\gamma \left( |_{L}\eta_{T-1}^{A_n}| > 2^dR\right) & \ge  1 -2 \delta, \\
\P^\gamma  \left( {N}^{A_n}(L,T) > d2^dR \right)  
& \ge   1 - \frac{2\delta^2}{1-\P^\gamma \left( 
|_{L}{\eta}_{T}^{A_n}|  > 2^dR \right) }
\ge 1-\delta.
\end{align*}
By lemmas~\ref{FKG1} and \ref{FKG2}, we obtain
\begin{align}
\P^\gamma \left( |_{L}{\eta}_{T-1}^{A_n}\cap \{0,\dots,L\}^d| > R\right) & \ge   1 -(2 \delta)^{2^{-d}}, \label{enhautbis}\\
\P^\gamma  \left( {N}_+^{A_n}(L,T) > R \right)  &\ge  1-\delta^{2^{-d}/d}. \label{cotebis}
\end{align}
But if $|_{L}{\eta}_{T-1}^{A_n}\cap \{0,\dots,L\}^d| > R$, 
\begin{itemize}
\item either $|_{L}\underline{\eta}_{T-1}^{A_n}\cap \{0,\dots,L\}^d| \ge  \sqrt R \ge N$, 
\item or there is one $x \in{}_{L}\underline{\eta}_{T-1}^{A_n}\cap \{0,\dots,L\}^d$ such that $\eta_{T-1}^{A_n}(x) \ge \sqrt R \ge J$.
\end{itemize} 
Using Equation \eqref{enhautbis}, either the choice we made for $N$ or the choice \eqref{choixJ} we made for $J$ leads to
\begin{equation}
 \P^\gamma \left( \exists x \in \{0,\dots,L\}^d  \quad _{L+2n}\underline{\eta}_{T-1+h}^{A_n} \supset x+A_n 
\right)  \ge  (1 -(2 \delta)^{2^{-d}})(1-\delta).
\end{equation}
Now, we can use \eqref{cotebis} starting from the translated $x+A_n$ occupied at time $T-1+h$. But if ${N}_+^{A_n}(L,T) > R$, 
\begin{itemize}
\item either the number of occupied sites on this face is larger than $\sqrt R \ge M$, 
\item or one of the occupied sites contains more than $\sqrt R \ge J$ particles.
\end{itemize} 
Using Equation \eqref{cotebis}, either the choice 
we made for $M$ or the choice \eqref{choixJ} we made for $J$ leads to
\begin{align*}
& \P^\gamma \left( \begin{array}{c}
\exists x \in \{L+n,\dots,2L+n\}\times \{0,\dots, 2L\}^{d-1}, \\
\exists t \in \{T-1+2h,\dots,2T-1+2h\} \\
 _{2L+2n}\eta_{t}^{A_n} \supset x+A_n
\end{array} \right)  \\
\ge & (1 -(2 \delta)^{2^{-d}})(1-\delta)(1-\delta^{2^{-d}/d})(1-\delta) \ge 1-\epsilon.
\end{align*}
Renaming $T+h$ by $T$ gives the result. 

\medskip

\underline{Suppose finally that $(H^c)$ and $(A_{\delta}^c)$ hold }\\
In other words, there exists $L$ such that 
\begin{align}
\label{rouge}
& \lim_{T \to +\infty}  \P^\gamma (|_{L}\eta_{T}^{A_n}| > 0)= \P^\gamma(\,_L\tau^{A_n} =+\infty )\ge 1-2\delta.
\end{align}
We see that this is the most favorable case for the branching random walk, because it can survive even in a finite box. We will see it is also the easiest case to deal with.
 
We fix then a large $L$ such that~\eqref{rouge} is satisfied.
There exists $h>0$ such that
$$\P^\gamma(\forall x \in \{-L,\dots,L\}^d , \; (L+n)e_1+A_n \subset \,_{2L+2n}\underline{\eta}^x_h)=r>0.$$
For $k \ge 0$, set $T_k=kh$ and  
$$B_k=\{\forall x \in \{-L,\dots,L\}^d , \; (L+n)e_1+A_n \subset \,_{2L+2n}\underline{\eta}^x_h \circ \theta_{T_k}\}.$$
The $B_k's$ are clearly independent. Moreover
$$ \{\,_L\tau^{A_n} =+\infty\} \cap B_k \subset \{(L+n)e_1+A_n \subset \,_{2L+2n}\underline{\eta}^{A_n}_{T_{k+1}}\}.$$
Thus, for any $K>1$, 
\begin{align*}
&\P^\gamma(\,_L\tau^{A_n} =+\infty, \; \forall t \in \{T_{K},\dots,T_{2K}\} \; (L+n)e_1+A_n \not\subset \,_{2L+2n}\underline{\eta}^{A_n}_{t}) \\
& \le  \P^\gamma(\,_L\tau^{A_n} =+\infty, \; \miniop{2K-1}{\cap}{k=K-1}B_k^c) \\
& \le  \P^\gamma\left(\miniop{2K-1}{\cap}{k=K-1}B_k^c\right)=(1-r)^{K+1}.
\end{align*}
We set now $2\delta=\varepsilon/2$, and take $K$ large enough to have $(1-r)^{K+1}<\varepsilon/2$, and
\begin{align*}
& \P^\gamma \left( 
\begin{array}{c}
\exists x \in \{L+n,\dots,2L+n\}\times \{0,\dots,2L\}^{d-1}, t \in \{T,\dots,2T\} \\
 _{2L+2n}\eta_{t}^{A_n} \supset x+A_n 
 \end{array}\right) \\
& \ge   \P^\gamma(\,_L\tau^{A_n} =+\infty, \; \exists t \in \{T_{K},\dots,T_{2K}\} \; (L+n)e_1+A_n \subset \,_{2L+2n}\underline{\eta}^{A_n}_{t}) \\
& \ge  \P^\gamma(\,_L\tau^{A_n} =+\infty)-\P^\gamma(\,_L\tau^{A_n} =+\infty, \; \forall t \in \{T_{K},\dots,T_{2K}\} \; (L+n)e_1+A_n \not\subset \,_{2L+2n}\underline{\eta}^{A_n}_{t}) \\
& \ge  1-2\delta-\varepsilon/2=1-\varepsilon.
 \end{align*}

One can see that we had to split the study into three cases and to make (not so) different proofs. This can seem odd, but we think it is unavoidable.
Note for instance that there is a similar disjunction in Yao and Chen~\cite{MR2946436}.

\def\refname{References}
\bibliographystyle{plain}


\end{document}